\documentclass[11pt,reqno]{amsart}
\usepackage{amsmath}
\usepackage{cases}
\usepackage{mathrsfs}
\usepackage{bbm}
\usepackage{amssymb}
\usepackage{amscd}
\usepackage{amsfonts,latexsym,amsmath,amsthm,amsxtra,mathdots,amssymb,latexsym,mathabx}
\usepackage[all,cmtip]{xy}
\RequirePackage{amsmath} \RequirePackage{amssymb}
\usepackage{color}
\usepackage{colordvi}
\usepackage{multicol}
\usepackage{hyperref}
\usepackage{mathtools}
\usepackage[margin=1in]{geometry}
\usepackage{xcolor}
\hypersetup{
    colorlinks,
    linkcolor={red!50!black},
    citecolor={blue!50!black},
    urlcolor={blue!80!black}
}
\usepackage{cite}

\newcommand{\bea}{\begin{eqnarray}}
\newcommand{\eea}{\end{eqnarray}}
\newcommand{\bna}{\begin{eqnarray*}}
\newcommand{\ena}{\end{eqnarray*}}

\numberwithin{equation}{section}

\theoremstyle{plain}
\newtheorem*{Theorem A}{Theorem A}
\newtheorem*{Theorem B}{Theorem B}
\newtheorem{problem}{Problem}[section]
\newtheorem{lemma}{Lemma}[section]
\newtheorem{theorem}{Theorem}[section]

\theoremstyle{definition}

\newtheorem{remark}{Remark}

\begin{document}

\title
[{Solutions to some sign change problems on the functions involving sums of divisors}] {Solutions to some sign change problems on the functions involving sums of divisors}

\author
[Y. Ding \& H. Pan \& Y.--C. Sun] {Yuchen Ding, Hao Pan and Yu--Chen Sun}

\address{(Yuchen Ding) School of Mathematical Science,  Yangzhou University, Yangzhou 225002, People's Republic of China}
\email{\tt ycding@yzu.edu.cn}

\address{(Hao Pan) School of Applied Mathematics, Nanjing University of Finance and Economics, Nanjing 210023, People’s Republic of China}
\email{\tt haopan79@zoho.com}

\address{(Yu--Chen Sun) Department of Mathematics and Statistics, University of Turku, Turku 20014 , Finland}
\email{\tt yucsun@utu.fi}



\begin{abstract}
In this note, we solve some sign change problems on the functions involving sums of divisors posed by Pongsriiam recently.
\end{abstract}

\maketitle
\section{Introduction}
Inspiring by the prior results of Erd\H os \cite{Erdos1,Erdos2}, Wang and Chen \cite{wangchen}, Pongsriiam \cite{Pongsriiam} investigated the arithmetic functions on generalized sums of divisors functions, i.e.,
$$
\sigma_s(n)=\sum_{d|n} d^s,
$$
where $s$ is a real number. He obtained various results in the same spirit of Newman's theorem \cite{Newman} which states that there are infinitely many $n$ such that
$$
\varphi(30n+1)<\varphi(30n),
$$
where $\varphi(n)$ is the Euler totient function. Newman's result is of particular interest because Jarden \cite[page 65]{Jarden} observed that
\begin{align}\label{0105-1}
\varphi(30n+1)>\varphi(30n)
\end{align}
for all $n\le 10^5$. The fact that there are infinite many $n$ making Eq. (\ref{0105-1}) true follows immediately from Dirichlet's theorem in arithmetic progressions. In other words, we know that
$$
\varphi(30n+1)-\varphi(30n)
$$
changes sign infinitely often, thanks to the theorems of Dirichlet and Newman. It contributes to another famous example of the Chebyshev bias phenomenon \cite{Chebyshev}.

From now on, let $\mathbb{N}$ and $\mathcal{P}$ be the sets of positive integers and primes, respectively. In Pongsriiam's article \cite{Pongsriiam}, he posed several open problems for further research. 

\begin{problem}[\cite{Pongsriiam}, Problem 3.7]\label{problem1}
Suppose $s>1$ and $a,b,c,d$ are positive integers. If
$$
\sigma_s(an+b)-\sigma_s(cn+d)
$$
changes sign infinitely often, is it true that $a=c$? If it is not true, assuming further that $s$ is large and $(a,b)=(c,d)=1$, can we conclude that $a=c$?
\end{problem}

This problem, as one can imagine, is motivated by Pongsriiam's result \cite[Theorem 2.10]{Pongsriiam}, which states that if $a,b,c,d$ are integers with $a>c>0$ and $b\ge d\ge 0$, then there exists $s_0>1$ such that
$$
\sigma_s(an+b)>\sigma_s(cn+d)
$$
for all $s\ge s_0$ and $n\ge 1$.

\begin{problem}[\cite{Pongsriiam}, Problem 3.8 (ii)]\label{problem2}
Is it true that
$$
\sum_{n\le K}\sigma(30n)>\sum_{n\le K}\sigma(30n+1)
$$
for all $K\in \mathbb{N}$?
\end{problem}

Pongsriiam \cite[Theorem 2.4]{Pongsriiam}  himself proved that
$$
\sigma(30n)-\sigma(30n+1)
$$
has infinitely many sign changes. Thus, Problem \ref{problem2}, if true, is certainly nontrivial.

\begin{problem}[\cite{Pongsriiam}, last two lines of Example 2.13]\label{problem3}
Are there infinitely many $m\in \mathbb{N}$ for which the inequalities 
$$
\sigma(2m+5)>\sigma(6m+17) \quad\quad \text{and} \quad\quad \sigma(5m+4)>\sigma(6m+7)
$$
simultaneously hold.
\end{problem}

This problem is motivated by Pongsriiam's result \cite[Theorem 2.11]{Pongsriiam}, where he proved that
$$
\sigma_s(2m+5)<\sigma_s(6m+17) \quad\quad \text{and} \quad\quad \sigma_s(5m+4)<\sigma_s(6m+7)
$$
for all $m\in \mathbb{N}$ and $s>3$.

The aim of our note is the following progress on the problems above.

\begin{theorem}\label{thm1}
The answer to Problem \ref{problem1} is negative.
\end{theorem}

\begin{theorem}\label{thm2}
The answer to Problem \ref{problem2} is positive providing that $K$ is sufficiently large.
\end{theorem}

\begin{theorem}\label{thm3}
The answer to Problem \ref{problem3} is positive.
\end{theorem}

\section{Proof of Theorem \ref{thm1}}
For the proof of Theorem \ref{thm1}, we need the following remarkable theorem of Dirichlet on primes in arithmetic progressions (see e.g. \cite[Page 34, Chapter 4]{Da}).

\begin{lemma}[Dirichlet]\label{lemma1}
Let $m$ and $q$ be integers such that $(m,q)=1$, then there are infinitely many primes $p\equiv m\pmod{q}$.
\end{lemma}

\begin{proof}[Proof of Theorem \ref{thm1}]
Let $s>1$ be any fixed number. It suffices to prove that there exists some positive odd integer $a$ (depending on $s$) such that
$$
\sigma_s(an+2)-\sigma_s((a+1)n+1)
$$
changes sign infinitely often. The proof goes as follows:

Note firstly that $(a,2)=1$. Then by Lemma \ref{lemma1}, there are infinitely many positive integers $n>1$ such that $an+2$ are primes. For these $n$, we have
$$
\sigma_s(an+2)=1+(an+2)^s
$$
and
$$
\sigma_s((a+1)n+1)\ge 1+((a+1)n+1)^s>1+(an+2)^s,
$$
from which it follows clearly that
$$
\sigma_s(an+1)-\sigma_s((a+1)n+1)<0.
$$

Now, suppose that $n=2k$ with $k\in \mathbb{Z}^{+}$, then
$$
(a+1)n+1=2(a+1)k+1.
$$
Since $(2(a+1),1)=1$, there are infinitely many $k$ such that $(a+1)n+1$ are primes. For these $n=2k$, we have
$$
\sigma_s(an+2)\ge 1+\left(\frac{an+2}{2}\right)^s+(an+2)^s
$$
and
$$
\sigma_s((a+1)n+1)=1+((a+1)n+1)^s.
$$
Thus, for these $n=2k$ we can deduce
\begin{align*}
\sigma_s(an+2)-\sigma_s((a+1)n+1)&\ge \left(\frac{an+2}{2}\right)^s+(an+2)^s-((a+1)n+1)^s\\
&>\left(1+\frac{1}{2^s}\right)a^sn^s-(a+2)^sn^s.
\end{align*}
Now, we are going to show that
\begin{align}\label{eq03-1}
(a+2)^s<\left(1+\frac{1}{2^{s+1}}\right)a^s
\end{align}
provided that $a$ is sufficiently large, from which our theorem would follow immediately.
Actually, Eq. (\ref{eq03-1}) is equivalent to
\begin{align*}
a>2\left(\left(1+\frac{1}{2^{s+1}}\right)^{1/s}-1\right)^{-1},
\end{align*}
which is clearly true for sufficiently large odd integer $a$.
\end{proof}

\section{Proof of Theorem \ref{thm2}}

The proof of Theorem \ref{thm2} follows from the following result \cite[III. 8, page 85]{Handbook}.

\begin{lemma}\label{lem2}
Let 
$f(x)$ be a polynomial with integer coefficients, of degree $n$, and such that $f(m)>0$ for all positive integers $m$. Then we have
$$
\sum_{m\le x}\sigma(f(m))=\frac{\beta a_n}{n+1}x^{n+1}+O\left(x^n(\log x)^n\right),
$$
where 
$$
\beta=\sum_{d=1}^{\infty}\frac{N(d)}{d^2}
$$
and $a_n$ is the coefficient of $x^n$ in $f(x)$. Here $N(m)$ denotes the number of solutions, not counting multiplicities, of the congruence $f(n)\equiv 0\pmod{m}$.
\end{lemma}

\begin{proof}[Proof of Theorem \ref{thm2}]
Note that for given $f$ the function $N(m)$ is multiplicative of $m$ by the Chinese reminder theorem. So we get
$$
\beta=\sum_{d=1}^{\infty}\frac{N(d)}{d^2}=\prod_{p}\left(1+\frac{N(p)}{p^2}+\frac{N(p^2)}{p^4}+\cdot\cdot\cdot\right),
$$
where the products extends over all the primes. 

Choose firstly $f(x)=30x$ in Lemma \ref{lem2}. Then for $p \nmid 30$ we clearly have $N\left(p^\alpha\right)=1$ for any $\alpha\in \mathbb{N}$.  It is also easy to see that
$N\left(p^\alpha\right)=p$
for $p=2,3,5$ and any $\alpha\in \mathbb{N}$. Thus, we have
$$
\beta=\prod_{p}\left(1+\frac{N(p)}{p^2}+\frac{N(p^2)}{p^4}+\cdot\cdot\cdot\right)=\frac{5}{3}\frac{11}{8}\frac{29}{24}\prod_{p~\!\!\nmid ~\!\!30}\left(1-p^{-2}\right)^{-1}=\frac{1595}{576}\prod_{p~\!\!\nmid ~\!\!30}\left(1-p^{-2}\right)^{-1}.
$$
Hence, we obtain
\begin{align}\label{eq16-1}
\sum_{n\le K}\sigma(30n)=\frac{1595}{576}\cdot15\prod_{p~\!\!\nmid ~\!\!30}\left(1-p^{-2}\right)^{-1}x^{2}+O\left(x\log x\right),
\end{align}
via Lemma \ref{lem2}.

Choose secondly $f(x)=30x+1$ in Lemma \ref{lem2}. Then for $p \nmid 30$ we have $N\left(p^\alpha\right)=1$ for any $\alpha\in \mathbb{N}$.  It is also easy to see that
$N\left(p^\alpha\right)=0$
for $p=2,3,5$ and any $\alpha\in \mathbb{N}$. Thus, we have
$$
\beta=\prod_{p}\left(1+\frac{N(p)}{p^2}+\frac{N(p^2)}{p^4}+\cdot\cdot\cdot\right)=\prod_{p~\!\!\nmid ~\!\!30}\left(1-p^{-2}\right)^{-1}.
$$
Hence, we obtain
\begin{align}\label{eq16-2}
\sum_{n\le K}\sigma(30n+1)=15\prod_{p~\!\!\nmid ~\!\!30}\left(1-p^{-2}\right)^{-1}x^{2}+O\left(x\log x\right),
\end{align}
again via Lemma \ref{lem2}. 

It can be concluded from Eqs. (\ref{eq16-1}) and (\ref{eq16-2}) that 
$$
\lim_{K\rightarrow\infty}\frac{\sum_{n\le K}\sigma(30n)}{\sum_{n\le K}\sigma(30n+1)}=\frac{1595}{576}>1,
$$
from which our theorem follows immediately.
\end{proof}

\section{Proof of Theorem \ref{thm3}}

Below, we shall prove a much more stronger result comparing with Theorem \ref{thm3}.  

\begin{theorem}\label{thm4}
Let $L_i(x)=a_ix+b_i,~H_j(x)=c_jx+d_j~(1\le i,j\le k)$ be linear functions with integer coefficients such that $a_i,c_j>0$ and 
$$
a_id_j\neq b_ic_j \quad (~\forall~1\le i,j\le k).
$$
Suppose that for any prime $p$ there is at least one integer $n_p$ such that $p$ divides none of the $L_i(n_p)$, then there are $\gg x(\log x)^{-k}$ integers $n\le x$ such that 
$$\sigma(H_j(n))>\max_{1\le i\le k}\sigma(L_i(n))  \quad\quad (~\forall~ 1\le j\le k),$$
where the constant implied by $\gg$ depends on the choices of the functions $L_i$ and $H_j$. 
\end{theorem}

Let $\Omega(n)$ be the number of all prime factors of $n$.
We need a sieve result of Heath--Brown \cite[Theorem 1]{HB}. 

\begin{lemma}[Heath--Brown]\label{lem3}
Let $L_i(x)=a_ix+b_i,~(1\le i\le k)$ be linear functions with integer coefficients and $a_i>0$. Suppose that for any prime $p$ there is at least one integer $n_p$ such that $p$ divides none of the $L_i(n_p)$, then there are $\gg x(\log x)^{-k}$ integers $n\le x$ such that
$$
\max_{1\le i\le k} \Omega(L_i(n))\le  G_k,
$$
where the constant $G_k$ is given by
$$
G_k=\left\lfloor \log_2\left\lfloor \frac{3k^2+4k+4}{2}\right\rfloor\right\rfloor.
$$ 
\end{lemma}

Our idea on the proof of Theorem \ref{thm4}, roughly speaking, is to choose some suitable integers $n$ such that $L_i(n)$ has `small' prime factors whereas $H_j(n)$ has  `many' prime factors. Let's implement our idea below.

\begin{proof}[Proof of Theorem \ref{thm4}]
Let $p_\ell$ be the $\ell^{\text{th}}$ prime. For any $h$ there exists some integer $n_\ell$ such that
\begin{align}\label{eq16-3}
p_\ell\nmid \prod_{1\le i\le k}L_i(n_\ell)
\end{align}
by the condition of the theorem. Let
$$
A=\max_{1\le i\le k}a_i, \quad  B=\max_{1\le j\le k}c_j \quad \text{and} \quad C=\max_{1\le i,j\le k}|a_id_j-b_ic_j|
$$
Suppose that $p_{\ell_1}$ is the smallest prime number strictly greater than $A+B+C+k$, then
there are $k$ strings of consecutive primes 
$$
p_{\ell_1+1},p_{\ell_1+2},...,p_{\ell_2}; \quad p_{\ell_2+1},p_{\ell_2+2},...,p_{\ell_3}; \quad \cdot\cdot\cdot \cdot\cdot\cdot; \quad p_{\ell_k+1},p_{\ell_k+2},...,p_{\ell_{k+1}}
$$
such that for any $1\le s\le k$ we have
\begin{align}\label{eq16-4}
\sum_{\ell=\ell_s+1}^{\ell_{s+1}}\frac{1}{p_\ell}>A2^{G_k+3}.
\end{align}
By the Chinese reminder theorem, the system of congruences
$$
\begin{cases}
n\equiv n_\ell\pmod{p_\ell}, \quad\quad \quad\quad (~\forall~1\le \ell\le \ell_1)\\
c_1n+d_1\equiv 0\pmod{p_\ell}, \quad (~\forall~\ell_1+1\le \ell\le \ell_2),\\
c_2n+d_2\equiv 0\pmod{p_\ell}, \quad (~\forall~\ell_2+1\le \ell\le \ell_3),\\
\qquad\vdots\\
c_kn+d_k\equiv 0\pmod{p_\ell}, \quad (~\forall~\ell_k+1\le \ell\le \ell_{k+1})
\end{cases}
$$
has exactly one solution $n\equiv n_0\pmod{P}$, where
$$
P=p_1\cdot\cdot\cdot p_{\ell_1} \prod_{1\le s\le k}\Big(\prod_{\ell_s+1\le \ell\le \ell_s+1}p_\ell\Big).
$$

Suppose now that $n=mP+n_0\le x$, then $m\le (x-n_0)/P$. We claim that 
$$
\widetilde{L_i}(m)=L_i(n)=L_i(mP+n_0)=a_iPm+(a_in_0+b_i) \quad (1\le i\le k)
$$
has no fixed prime divisor, that is to say, for any prime $p$ there is at least one integer $\widetilde{m}_p$ such that $p$ divides none of the $\widetilde{L_i}(\widetilde{m}_p)$. In fact, we choose $\widetilde{m}_p=0$ for $p\le p_{\ell_1}$. Then by the system of congruences and Eq. (\ref{eq16-3}), we have
$$
\widetilde{L_i}(0)\equiv a_in_0+b_i\equiv a_in_\ell+b_i\not\equiv 0\pmod{p_\ell}
$$
providing that $\ell\le \ell_1$.
For any $p_{\ell_1+1}\le p\le p_{\ell_{k+1}}$, we choose again $\widetilde{m}_p=0$. Note that there is some $1\le j\le k$ such that
$$
c_jn_0+d_j\equiv 0\pmod{p}
$$
by the system of congruences again. Then
$$
n_0\equiv -c_j^{-1}d_j \pmod{p}
$$
due to the fact that $p> B$.
Thus, for this prime $p$ we have
$$
\widetilde{L_i}(0)\equiv a_in_0+b_i\equiv a_i(-c_j^{-1}d_j)+b_i \not\equiv 0\pmod{p}
$$
since $p>C$ and
$$
a_id_j\neq b_ic_j \quad (~\forall~1\le i,j\le k).
$$
Finally, for any prime $p>p_{\ell_{k+1}}$, we have $(p,a_iP)=1$. Then 
there is exactly one solution to the congruence
$$
\widetilde{L_i}(m)\equiv 0 \pmod{p}
$$
for each $\widetilde{L_i}(m)$. Hence, the congruence 
$$
\widetilde{L_1}(m)\cdot\cdot\cdot\widetilde{L_k}(m)\equiv 0 \pmod{p}
$$
has at most $k(<p)$ solutions, which completes the proof of our claim.

By our claim and Lemma \ref{lem3} there are 
$$
\gg \frac{x-n_0}{P}\left(\log \frac{x-n_0}{P}\right)^{-k}\gg_k x(\log x)^{-k}
$$ 
integers $m\le (x-n_0)/P$ such that
$$
\max_{1\le i\le k}\Omega\left(L_i(n)\right)=\max_{1\le i\le k} \Omega\left(\widetilde{L}_i(m)\right)\le G_k.
$$
For such integers $m$, i.e., $n=mP+n_0$, we have
\begin{align}\label{eq16-5}
\max_{1\le i\le k}\sigma\left(L_i(n)\right)\le 2An\cdot 2^{G_k}=2^{G_k+1}An,
\end{align}
provided that $n$ is sufficiently large.
On the other hand, for such integers $n=mP+n_0$ we have
\begin{align}\label{eq16-6}
\sigma(H_j(n))\ge \sum_{\ell=\ell_j+1}^{\ell_{j+1}}\frac{c_jn+d_j}{p_\ell}\ge \frac{n}{2}A2^{G_k+3}=2^{G_k+2}An
\end{align}
via Eq. (\ref{eq16-4}), provided that $n$ is sufficiently large.
Our theorem now follows directly from Eqs. (\ref{eq16-5}) and (\ref{eq16-6}).
\end{proof}

\begin{proof}[Proof of Theorem \ref{thm3}]
We take $L_1(x)=6x+17$, $L_2(x)=6x+7$, $H_1(x)=2x+5$ and $H_2(x)=5x+4$ in Theorem \ref{thm4}.
It can be easily checked that $L_1$ and $L_2$ have no fixed prime divisor and
$$
6\cdot 5\neq 2\cdot17, \quad 6\cdot 4\neq 5\cdot 17, \quad  6\cdot 5\neq 2\cdot7, \quad 6\cdot 4 \neq 5\cdot 7.
$$
Thus, we complete the proof of Theorem \ref{thm3}.
\end{proof}

\begin{remark}\label{remark1}
It is worth mentioning that some corresponding results can be established in the same way if one replaces the function $\sigma(\cdot)$ by $\sigma_s(\cdot)$ in Theorem \ref{thm4} with $s\le 1$ (but not applicable for $s>1$ since the corresponding inequality for Eq. (\ref{eq16-4}) will no longer hold). These variants of Theorem \ref{thm4}, in principle, provide partial answers to some other problems of Pongsriiam (see \cite[Problem 3.2]{Pongsriiam} and \cite[Problem 3.3]{Pongsriiam}). We do not expand  details of these variants for the simplicity of our note as far as possible.
\end{remark}

\bigskip

\noindent{\bf Acknowledgements.}

The first named author is supported by National Natural Science Foundation of China (Grant No. 12201544), Natural Science Foundation of Jiangsu Province in China (Grant No. BK20210784) and China Postdoctoral Science Foundation (Grant No. 2022M710121). He is also supported by the foundations (Grant No. JSSCBS20211023).

The second named author is supported by  the National Natural Science Foundation of China (Grant No. 12071208).

The third named author was supported by UTUGS funding and was working in the Academy of Finland project No. $333707$.

\end{document}